\documentclass[10pt]{amsart}
\usepackage{amsmath}
\usepackage{amssymb}
\usepackage{amsfonts}
\usepackage{mathrsfs}
\usepackage{amscd}
\usepackage{enumerate}
\usepackage{amscd}
\usepackage{enumerate}
\usepackage[sans]{dsfont} 
\DeclareMathAlphabet{\mathpzc}{OT1}{pzc}{m}{it} 

\numberwithin{equation}{section}       
\numberwithin{figure}{section}       

\theoremstyle{plain}

\newtheorem{prop}{Proposition}[section]

\newtheorem{lemm}[prop]{Lemma}
\newtheorem{fact}[prop]{Fact}
\newtheorem{theoalph}{Theorem}

\theoremstyle{definition}
\newtheorem{defi}[prop]{Definition}

\theoremstyle{remark}
\newtheorem{rema}[prop]{Remark}

\newtheoremstyle{citing}
  {3pt}
  {3pt}
  {\itshape}
  {}
  {\bfseries}
  {.}
  {.5em}
  {\thmnote{#3}}

\theoremstyle{citing}
\newtheorem*{generic}{}

\newcommand{\C}{\mathbb{C}}

\newcommand{\R}{\mathbb{R}}

\newcommand{\cL}{\mathcal{L}}
\newcommand{\cM}{\mathcal{M}}

\newcommand{\cV}{\mathcal{V}}
\newcommand{\cW}{\mathcal{W}}

\newcommand{\fD}{\mathfrak{D}}

\newcommand{\sA}{\mathscr{A}}

\newcommand{\sG}{\mathscr{G}}

\newcommand{\sP}{\mathscr{P}}

\newcommand{\sW}{\mathscr{W}}

\newcommand{\tB}{\widetilde{B}}

\newcommand{\teta}{\widetilde{\teta}}

\newcommand{\tvarphi}{\widetilde{\varphi}}

\renewcommand{\=}{ : = }

\DeclareMathOperator{\diam}{diam}

\DeclareMathOperator{\dist}{dist}

\DeclareMathOperator{\Crit}{Crit} 
\DeclareMathOperator{\CV}{CV} 

\newcommand{\CC}{\overline{\C}}

\DeclareMathOperator{\Per}{Per}

\newcommand{\hcL}{{\widehat{\mathcal{L}}}}

\DeclareMathOperator{\dom}{dom}

\newcommand{\eps}{\varepsilon}

\newcommand{\1}{\pmb{1}}

\begin{document}

\title[Large deviation for H\"{o}lder continuous potential]{Large deviation principles of one-dimensional maps for H\"{o}lder continuous potentials}

\author{Huaibin Li}
\thanks{The author was partially supported by
the National Natural Science Foundation of China (Grant No. 11101124) and
the FONDECYT grant 3110060 of Chile.}
\address{Huaibin Li, Facultad de Matem{\'a}ticas, Pontificia Universidad Cat{\'o}lica de Chile, Avenida Vicu{\~n}a Mackenna~4860, Santiago, Chile}
\email{matlihb@gmail.com}

\begin{abstract}
We show some level-$2$ large deviation principles for real and complex one-dimensional maps satisfying a weak form of hyperbolicity. More precisely, we prove a large deviation principle for the distribution of iterated preimages,
periodic points, and Birkhoff averages.
\end{abstract}

\maketitle

\section{Introduction}
We consider large deviations for dynamical systems. It is well known that for uniformly
hyperbolic dynamical systems the large deviation principle holds, and the rate functions
are characterized by the differences between the metric entropies and the sums of positive
Lyapunov exponents of invariant probability measures, see for example~\cite{Kif90, You90} and references therein.  For nonuniformly
hyperbolic dynamical systems there are some known results, see for example~\cite{AraPac06, ChuTak12, Com09, ComRiv11,CliThoYam13, KelNow92, Mel09b, MelNic08,  PolSha96, PolSha09b, PolSha09a, PolShaYur98, ReyYou08, XiaFu07}, and references therein. See also the survey paper of Denker~\cite{Den96}.

This paper is devoted to the study of level-$2$ large deviation principles for H\"{o}lder continuous potentials of real or complex one-dimensional  maps, under weak hyperbolicity  assumptions.  The first key observation is a result obtained by the author and Rivera-Letelier in~\cite{LiRiv12one} that for a real or complex one-dimensional map satisfying such weak hyperbolicity  assumptions, every H\"{o}der continuous potential has a unique equilibrium state, see Lemma~\ref{l:uniqueeq}. This, together with various formulas to compute the pressure function,  allows us to apply~\cite[Theorem~C]{ComRiv11}, which is a
variant of a general result of Kifer in~\cite[Theorem 3.4]{Kif90}, to obtain a full level-$2$ large deviation
principles for sequences of measures associated to periodic points, iterated preimages, and
Birkhoff averages.

We now proceed to describe our results in more detail.
\subsection{Equilibrium states and level-2 large deviation priciple}
Let~$(X, \dist)$ be a compact metric space and $T:X\to X$ a continuous map.
Denote by~$\cM(X)$ the space of Borel probability measures on~$X$ endowed with the weak* topology, and by~$\cM(X, T)$ the subspace of~$\cM(X)$ of those measures that are invariant by~$T$.
For each measure~$\nu$ in~$\cM(X, T)$, denote by~$h_{\nu}(T)$ the \emph{measure-theoretic entropy} of~$\nu$.
For a continuous function $\varphi: X \to \R$, denote by~$P(T,\varphi)$ the \emph{topological pressure of $T$ for the potential $\varphi$}, defined by
\begin{equation}
\label{e:variational principle}
P(T, \varphi)
\=
\sup\left\{h_\nu(T) + \int_X \varphi \ d\nu : \nu\in \cM(X, T)\right\}.
\end{equation}
A measure~$\nu$ in~$\cM(X, T)$ is called an \emph{equilibrium state of $T$ for the potential~$\varphi$}, if the supremum in~\eqref{e:variational principle} is attained at~$\nu$.

Recall that a
sequence  Borel probability measures $(\Omega_n)_{n\ge 1}$  on
$\cM(X)$ is said to satisfy a \emph{large deviation principle} if there exists a lower semi-continuous function called a
\emph{rate function $I:\cM(X)\to [0,+\infty]$,} such that for every closed subset $\mathcal {F}$ of
$\cM(X)$ we have
$$\limsup_{n\rightarrow +\infty}\frac{1}{n}\log \Omega_n(\mathcal {F})\le -\inf_\mathcal {F} I,$$ and such that for every
open subset $\mathcal {G}$ of $\cM(X)$ we have,
$$\liminf_{n\rightarrow +\infty} \frac{1}{n}\log \Omega_n(\mathcal {G})\ge -\inf_{\mathcal
{G}}I. $$

\begin{defi}
Let~$(X, \dist)$ be a compact metric space and $T:X\to X$ be a continuous map. Let~$\varphi: X\rightarrow \mathbb{R}$ be a
H\"{o}lder continuous function.
Moreover, for each integer $n\ge 1$ let $F_n :
X\rightarrow \cM(X)$ be the function defined by
$$F_n(x):=\frac{1}{n}\sum_{i=0}^{n-1}\delta_{T^i(x)}=\frac{1}{n}\left(\delta_x+\delta_{T(x)}+\cdots+\delta_{T^{n-1}(x)}\right),$$
and let $S_n(\varphi): X\rightarrow \R$ be defined by
$$S_n(\varphi)(x):=\sum_{i=0}^{n-1}\varphi\circ T^i(x)=\varphi(x)+\varphi(T(x))\cdots+\varphi(T^{n-1}(x)). $$
\begin{itemize}
\item {\bf Iterated preimages:} Given $x_0\in X,$ for each integer
$n\ge 1$  put $$\Theta_n(x_0):=\sum_{y\in T^{-n}(x_0)}\frac{\exp
(S_n(\varphi)(y))}{\sum_{y'\in T^{-n}(x_0)}\exp
(S_n(\varphi)(y'))}\delta_{F_n(y)}. $$
\item{\bf Periodic points:} Letting  $P_n:= \{p \in X: T^n(p) = p\},$
put $$\Theta_n:=\sum_{p\in P_n}\frac{\exp
(S_n(\varphi)(p))}{\sum_{p'\in P_n}\exp
(S_n(\varphi)(p'))}\delta_{F_n(p)}. $$
\item {\bf Birkhoff averages:} Assume there is a unique equilibrium state $\nu_\varphi$ of $T$ for the potential $\varphi$, put $\Sigma_n:= F_n[\nu_\varphi]$($ \emph{i.e.}$ the
image measure of $\nu_\varphi$ by $F_n$).
\end{itemize}

\noindent We say the \emph{level-$2$ large deviation principle holds for $(T, \varphi)$ with respect to iterated preimages (resp. periodic orbits, Birkhoff averages)} if there exist a unique equilibrium state  $\nu_\varphi$ of $T$ for the
potential $\varphi$, and if the
sequences $(\Theta_n(x_0))_{n\ge 1}$ (resp.
$(\Theta_n)_{n\ge 1}$, $(\Sigma_n)_{n\ge 1}$)
converges
to~$\delta_{\nu_\varphi}$ in the weak$^*$ topology and satisfies a
large deviation principle.
\end{defi}

\subsection{Statements of results}
We state~$2$ results, one in the real setting and the other in the complex setting.
To simplify the exposition, they are formulated in a more restricted situation than what we are able to handle, see the Main Theorem in~\S\ref{s:a reduction} for a more general formulation of our results.

Given a compact interval~$I$, a smooth map~$f : I \to I$ is \emph{non-degenerate}, if the number of points at which the derivative of~$f$ vanishes is finite, and if for every such point there is a higher order derivative of~$f$ that is non-zero.
\begin{theoalph}
\label{t:real}
Let~$I$ be a compact interval and let $f: I \to I$ be a topologically exact non-degenerate smooth map.
Assume~$f$ has only hyperbolic repelling periodic points, and for each critical value~$v$ of~$f$ we have
$$ \lim_{n\to +\infty} |Df^n(v)| = +\infty. $$
Then for every H\"{o}lder continuous potential $\varphi: I\rightarrow \R$, the level-$2$ large deviation principle holds for $(f, \varphi)$ with respect to iterated preimages, periodic orbits, and Birkhoff averages, respectively, with the rate function $I^\varphi:
\cM(I)\rightarrow [0, +\infty]$ given by the following
$$I^\varphi(\mu)=\begin{cases}P(f,\varphi)-\int_I \varphi d\mu- h_\mu(f)
\,\,\,\,\,\,\mbox{if}\,\, \mu\in \cM(I, f);
\\ +\infty \,\,\hspace{35mm} \mbox{if}\,\, \mu\in
\cM(I)\setminus\cM(I, f).\end{cases}$$
Furthermore, for each convex open subset $\mathscr{G}$ of $\cM(I)$
containing some invariant measure we have $\inf_{\sG} I^\varphi =
\inf_{\overline{\sG}}I^\varphi,$ and
$$\lim_{n\rightarrow +\infty}\frac{1}{n}\log\Theta_n(\sG)=\lim_{n\rightarrow +\infty}\frac{1}{n}\log\Theta_n(z_0)(\sG)
=\lim_{n\rightarrow +\infty}\frac{1}{n}\log\Sigma_n(\sG)
=-\inf_{\sG}I^\varphi. $$
Moreover, the above expression remains true replacing $\sG$ by
$\overline{\sG}$.
\end{theoalph}

To state our main result in the complex setting, for each complex rational map~$f$ denote by~$\Crit(f)$ the set of critical points of~$f$, and by~$J(f)$ the Julia set of~$f$.
\begin{theoalph}
\label{t:complex}
Let~$f$ be one of the following:
\begin{enumerate}[1.]
\item
An at most finitely renormalizable complex polynomial of degree at least~$2$, without neutral cycles and such that for each critical value~$v$ of~$f$ we have
$$ \lim_{n\rightarrow +\infty}|Df^n(v)|=+\infty; $$
\item
A complex rational map of degree at least~$2$, without parabolic cycles and such that for every critical value~$v$ of~$f$ we have
$$ \sum_{n = 1}^{+\infty} \frac{1}{|Df^n(v)|} < +\infty. $$
\end{enumerate}
Then for every H\"{o}lder continuous potential $\varphi: J(f)\rightarrow \R$, the level-$2$ large deviation principle holds for $(f, \varphi)$ with respect to iterated preimages, periodic orbits, and Birkhoff averages, respectively, with the rate function $I^\varphi:
\cM(J(f))\rightarrow [0, +\infty]$ given by the following
$$I^\varphi(\mu)=\begin{cases}P(f,\varphi)-\int_{J(f)}\varphi d\mu- h_\mu(f)
\,\,\,\,\,\,\mbox{if}\,\, \mu\in \cM(J(f), f);
\\ +\infty \,\,\hspace{35mm} \mbox{if}\,\, \mu\in
\cM(J(f))\setminus\cM(J(f), f).\end{cases}$$
Furthermore, for each convex open subset $\mathscr{G}$ of $\cM(J(f))$
containing some invariant measure we have $\inf_{\sG} I^\varphi =
\inf_{\overline{\sG}}I^\varphi,$ and
$$\lim_{n\rightarrow +\infty}\frac{1}{n}\log\Theta_n(\sG)=\lim_{n\rightarrow +\infty}\frac{1}{n}\log\Theta_n(z_0)(\sG)
=\lim_{n\rightarrow +\infty}\frac{1}{n}\log\Sigma_n(\sG)
=-\inf_{\sG}I^\varphi. $$
Moreover, the above expression remains true replacing $\sG$ by
$\overline{\sG}$.
\end{theoalph}
Recall that for an integer~$s \ge 2$, a complex polynomial~$f$ is \emph{renormalizable of period~$s$}, if there are Jordan disks~$U$ and~$V$ in~$\C$, such that~$\overline{U} \subset V$ and such that the following hold:
\begin{itemize}
\item $f^s: U\to V$ is proper of degree at least~$2$;
\item The set $\{z\in U: f^{sn}(z)\in U \text{ for all } n=1,2,\ldots\}$ is a connected and it is strictly contained in~$J(f)$;
\item For each critical point~$c$ of~$f$ there exists at most one~$j$ in~$\{0,1,\ldots, s-1\}$ such that~$c$ is~$f^j(U)$.
\end{itemize}
We say that~$f$ is \emph{infinitely renormalizable} if there are infinitely many integers~$s \ge 2$ such that~$f$ is renormalizable of period~$s$.

\subsection{Organization of the proof}
\label{ss:organization}
The paper is organized as follows.
In~\S\ref{s:a reduction} we state a version of Theorems~\ref{t:real} and~\ref{t:complex} that holds for a more general class of maps; it is stated as the ``Main Theorem''.
After deriving Theorems~\ref{t:real} and~\ref{t:complex} from the Main Theorem and known results in~\S\ref{ss:proof of real and complex}, we sate our main technical result as the ``Key Lemma'' in~\S\ref{ss:a reduction}, where we also derive the Main Theorem from it.
In \S\ref{s:proofkeylemma}, we  give the proof of the Key Lemma.

\subsection{Acknowledgments}
\label{ss:Acknowledgements}
The main idea of this paper came to the author after several discussions with Juan Rivera-Letelier on his a joint work with Henri Comman, see~\cite{ComRiv11}. I am very grateful to him for those stimulating conversations  and for his useful comments and corrections to an earlier version of this paper.

\section{A reduction}
\label{s:a reduction}
We start this section stating a version of Theorems~\ref{t:real} and~\ref{t:complex} that holds for a more general class of maps; it is stated as the ``Main Theorem''.
In~\S\ref{ss:proof of real and complex} we derive Theorems~\ref{t:real} and~\ref{t:complex} as a direct consequence of the Main Theorem and known results.
In~\S\ref{ss:a reduction} we state our more general result as the ``Key Lemma'', and we derive the Main Theorem from it.

To state our main theorem, we introduce a class of interval maps that includes non-degenerate smooth maps as special cases.
Let~$I$ be a compact interval of $\R$.
For a differentiable map~$f : I \to I$, a point of~$I$ is \emph{critical} if the derivative of~$f$ vanishes at it.
Denote by~$\Crit(f)$ the set of critical points of~$f$.
A differentiable interval map~$f : I \to I$ is \emph{of class~$C^3$ with non-flat critical points}, if it has a finite number of critical points and if:
\begin{itemize}
\item
The map~$f$ is of class~$C^3$ outside $\Crit(f)$;
\item
For each critical point~$c$ of~$f$ there exists a number $\ell_c>1$ and diffeomorphisms~$\phi$ and~$\psi$ of~$\R$ of class~$C^3$, such that $\phi(c)=\psi(f(c))=0$, and such that on a neighborhood of~$c$ on~$I$, we have
$$ |\psi\circ f| = |\phi|^{\ell_c}. $$
\end{itemize}
Note that each smooth non-degenerate interval map is of class~$C^3$ with non-flat critical points, and that each map of class~$C^3$ with non-flat critical points is continuously differentiable.

For an interval map of class~$C^3$ with non-flat critical points~$f$, denote by~$\dom(f)$ the interval on which~$f$ is defined, and denote by~$\dist$ the distance on~$\dom(f)$ induced by the norm distance on~$\R$.
On the other hand, for a complex rational map~$f$ we use~$\dom(f)$ to denote the Riemann sphere~$\CC$, which we endow with the spherical metric, that we also denote by~$\dist$.
In both, the real and complex setting, for a subset~$W$ of~$\dom(f)$ we use~$\diam(W)$ to denote the diameter of~$W$ with respect to~$\dist$.

\begin{defi}
Let~$f$ be either a complex rational map, or an interval map of class~$C^3$ with non-flat critical points.
The \emph{Julia set~$J(f)$ of~$f$} is the complement of the largest open subset of~$\dom(f)$ on which the family of iterates of~$f$ is normal.
\end{defi}

If~$f$ is a complex rational map of degree at least~$2$, then~$J(f)$ is a perfect set that is equal to the closure of repelling periodic points.
Moreover, $J(f)$ is completely invariant and~$f$ is topologically exact on~$J(f)$.
We denote by~$\sA_{\C}$ the collection of all rational maps of degree at least~$2$.

In contrast with the complex setting, the Julia set of an interval map of class~$C^3$ with non-flat critical points might be empty, reduced to a single point, or might not be completely invariant.
In that follows, we denote by~$\sA_{\R}$ the collection of interval maps of class~$C^3$ with non-flat critical points, whose Julia set contains at least~$2$ points and is completely invariant. Note that if~$f : I \to I$ is a non-degenerate smooth map that is topologically exact, then~$J(f) = I$ and~$f$ is in~$\sA_{\R}$.
On the other hand, if~$f$ is an interval map in~$\sA_{\R}$ that is topologically exact on~$J(f)$, then~$J(f)$ has no isolated points.
For more background on the theory of Julia sets, see for example~\cite{dMevSt93} for the real setting, and~\cite{CarGam93,Mil06} for the complex setting.

Throughout the rest of this article we put~$\sA \= \sA_{\R}\cup \sA_{\C}$ and for each~$f$ in~$\sA$ we restrict the action of~$f$ to its Julia set.
In particular, the topological pressure of~$f$ is defined through measures supported on~$J(f)$ and equilibrium states are supported on~$J(f)$.

\begin{defi}
For~$\beta> 0$, a map~$f$ in~$\sA$ satisfies the \emph{Polynomial Shrinking Condition with exponent $\beta$}, if there exist constants~$\rho_0 > 0$ and~$C_0 > 0$ such that for every~$x$ in~$J(f)$, every integer $n \ge 1$, and every connected component~$W$ of $f^{-n}(B(x, \rho_0))$, we have
$$ \diam(W) \le C_0 n^{-\beta}. $$
\end{defi}

\begin{generic}[Main Theorem]\label{mtm:polyshrink}
Let~$\beta > 1$, and let~$f$ be a map in~$\sA$ satisfying the Polynomial Shrinking Condition with exponent~$\beta$.
Suppose furthermore in the real case that~$f$ is topologically exact on its Julia set.
Then for each~$\alpha$ in~$(\beta^{-1}, 1]$, and every H\"{o}lder continuous potential $\varphi: J(f)\to \R$ of exponent~$\alpha$, the level-$2$ large deviation principle holds for $(f, \varphi)$ with respect to iterated preimages, periodic orbits, and Birkhoff averages, respectively, with the rate function $I^\varphi:
\cM(J(f))\rightarrow [0, +\infty]$ given by the following
$$I^\varphi(\mu)=\begin{cases}P(f,\varphi)-\int_{J(f)}\varphi d\mu- h_\mu(f)
\,\,\,\,\,\,\mbox{if}\,\, \mu\in \cM(J(f), f);
\\ +\infty \,\,\hspace{35mm} \mbox{if}\,\, \mu\in
\cM(J(f))\setminus\cM(J(f), f).\end{cases}$$
Furthermore, for each convex open subset $\mathscr{G}$ of $\cM(J(f))$
containing some invariant measure we have $\inf_{\sG} I^\varphi =
\inf_{\overline{\sG}}I^\varphi,$ and
$$\lim_{n\rightarrow +\infty}\frac{1}{n}\log\Theta_n(\sG)=\lim_{n\rightarrow +\infty}\frac{1}{n}\log\Theta_n(z_0)(\sG)
=\lim_{n\rightarrow +\infty}\frac{1}{n}\log\Sigma_n(\sG)
=-\inf_{\sG}I^\varphi. $$
Moreover, the above expression remains true replacing $\sG$ by
$\overline{\sG}$.
\end{generic}

\begin{rema}In real case, our arguments and results can be generalized without change to the following. Let $f$ be an interval map in $\sA_\R$ that is topologically exact on its Julia set, and let $\varphi:J(f)\to \R$ be a H\"{o}lder continuous function that is hyperbolic for $f$. If there is a dense vector subspace $\sW$ of H\"{o}lder continuous functions defined on $J(f)$ such that for every $\psi\in \sW$, the function~$\varphi+\psi$ is hyperbolic for~$f$, then the level-$2$ large deviation principle holds for $(f, \varphi)$ with respect to iterated preimages, periodic orbits, and Birkhoff averages, respectively, with the rate function $I^\varphi:
\cM(J(f))\rightarrow [0, +\infty]$ given by the following
$$I^\varphi(\mu)=\begin{cases}P(f,\varphi)-\int_{J(f)}\varphi d\mu- h_\mu(f)
\,\,\,\,\,\,\mbox{if}\,\, \mu\in \cM(J(f), f);
\\ +\infty \,\,\hspace{35mm} \mbox{if}\,\, \mu\in
\cM(J(f))\setminus\cM(J(f), f).\end{cases}$$
Furthermore, for each convex open subset $\mathscr{G}$ of $\cM(J(f))$
containing some invariant measure we have $\inf_{\sG} I^\varphi =
\inf_{\overline{\sG}}I^\varphi,$ and
$$\lim_{n\rightarrow +\infty}\frac{1}{n}\log\Theta_n(\sG)=\lim_{n\rightarrow +\infty}\frac{1}{n}\log\Theta_n(z_0)(\sG)
=\lim_{n\rightarrow +\infty}\frac{1}{n}\log\Sigma_n(\sG)
=-\inf_{\sG}I^\varphi. $$
Moreover, the above expression remains true replacing $\sG$ by
$\overline{\sG}$.
\end{rema}

\subsection{Proofs of Theorems~\ref{t:real} and~\ref{t:complex} assuming the Main Theorem}
\label{ss:proof of real and complex}
First, we recall the definition of ``backward contracting property'' which was first introduced in~\cite{Riv07} for the complex case, and in~\cite{BruRivShevSt08} for the real case.
For each map~$f$ in~$\sA$, put
$$ \CV(f) \= f(\Crit(f))
\text{ and }
\Crit'(f) \= \Crit(f)\cap J(f). $$
For a subset~$V$ of~$\dom(f)$, and an integer~$m \ge 1$, each connected component of~$f^{-m}(V)$ is a \emph{pull-back of by~$f^m$}. A pull-back $W$ of $V$ by $f^n$ is \emph{diffeomorphic} if $f^n$
maps $W$ diffeomorphically onto a connected component of $V$, and it is \emph{non-diffeomorphic} otherwise.

For every~$c$ in~$\Crit(f)$ and $\delta >0$, denote by~$\tB(c,\delta)$ the pull-back of~$B(f(c),\delta)$ by~$f$ that contains~$c$.

\begin{defi}
  Given a constant $r>1$, a map~$f$ in~$\sA$ is \emph{backward contracting with constant $r$}, if there exists $\delta_{0}>0$ such that for every~$c$ in~$\Crit'(f)$, every~$\delta$ in~$(0, \delta_0)$, every integer $n\ge 0$, and every component~$W$ of $f^{-n}(\tB(c,r\delta)),$  we have that
$$ \dist (W, \CV(f)) \le \delta
\text{ implies }
\diam(W) < \delta. $$
If for each~$r > 1$ the map~$f$ is backward contracting with constant~$r$, then~$f$ is \emph{backward contracting}.
\end{defi}

A map~$f$ in~$\sA$ is \emph{expanding away from critical points}, if for every neighborhood~$V$ of~$\Crit'(f)$ the map~$f$ is uniformly expanding on the set
$$ K(V)
\=
\{z\in J(f): f^i(z)\not\in V \text{ for all } i\ge 0\}. $$
In other words, there exist $C>0$ and  $\lambda>1$ such that for every~$z$ in~$K(V)$ and~$n\ge 0$, we have $|Df^n(z)|\ge C\lambda^n$.

\begin{fact}[\cite{RivShe1004}, Theorem~A]
\label{f:Polyshrink}
For each map~$f$ in~$\sA$ and each $\beta>0,$ there exists $r>1$ such that the following property holds.
If~$f$ is backward contracting with constant~$r$ and is expanding away from critical points, then~$f$ satisfies the Polynomial Shrinking
Condition with exponent~$\beta$.
\end{fact}

\begin{proof}[Proof of Theorem~\ref{t:real}]
By~\cite[Theorem~$1$]{BruRivShevSt08}, the map~$f$ is backward
contracting and by Ma\~{n}e's theorem~$f$ is expanding away from
critical points, see for example~\cite{dMevSt93}.
Then Fact~\ref{f:Polyshrink} implies that for each~$\beta > 1$ the map~$f$ satisfies the Polynomial Shrinking Condition with exponent~$\beta$.
So the desired assertions are  direct consequences of the Main Theorem.
\end{proof}

\begin{proof}[Proof of Theorem~\ref{t:complex}]
By either~\cite[Theorem~A]{LiShe10b} or~\cite[Theorem~A]{Riv07}, the map~$f$ is backward contracting, and by either~\cite{KozvSt09} or~\cite[Corollary~$8.3$]{Riv07}, it is expanding away from critical points.
Then Fact~\ref{f:Polyshrink} implies that for each~$\beta > 1$ the map~$f$ satisfies the Polynomial Shrinking Condition with exponent~$\beta$.
So the desired assertions are  direct consequences of the Main Theorem.
\end{proof}

\subsection{A reduction}
\label{ss:a reduction}
In this subsection we prove the Main Theorem assuming the following key lemma, whose proof occupies the rest of this article.

Given an interval map~$f$ in~$\sA_\R$, a continuous function~$\varphi:J(f)\to \R$ is said to be \emph{hyperbolic for~$f$}, if for some integer~$n \ge 1$ the function $S_n(\varphi)\=\sum_{i=0}^{n-1}\varphi\circ f^{i}$ satisfies
$\sup_{J(f)} \frac{1}{n} S_n(\varphi) < P(f, \varphi).$

Let $f$ be a  map in~$\sA.$ Given a measurable function $g: J(f)\to [0,+\infty),$ a probability measure $\mu$ supported on $J(f)$ is \emph{$g$-conformal for $f$}, if for every measurable set $A\subset J(f)$ on which $f$ is injective, we have $$\mu(f(A))=\int_A g d\mu.$$

\begin{lemm}\label{l:uniqueeq}
Let $f$ be a map in $\sA$ that is topologically exact on its Julia set.  Then for every H\"{o}lder continuous function $\varphi:J(f)\to \R$ that is hyperbolic for $f$, there exist an atom-free $\exp(P(f,\varphi)-\varphi)$-conformal measure $\mu$ for $f$, and a unique equilibrium state $\nu$ of $f$ for the potential $\varphi.$ Moreover,~$\nu$ is absolutely continuous with respect to $\mu$.
\end{lemm}
\begin{proof}
In the complex case, the assertions are a direct consequence  of~\cite[Theorem]{DenUrb91e} and of~\cite[Proposition~3.1]{InoRiv12}.
In the real setting, the assertions follow directly from~\cite[Theorem~A]{LiRiv12one} and~\cite[Theorem~6]{Dob13} or~\cite[Theorem~B]{LiRiv12one}.
\end{proof}

\begin{generic}[Key Lemma]
Let $f$ be an interval map in~$\sA_\R$ that is topologically exact on its Julia set $J(f).$ Then for every H\"{o}lder continuous function $\varphi:J(f)\to \R$ that is hyperbolic for $f$, the following properties hold:
\begin{enumerate}[1.]
\item
For every point $x_0$ in $J(f)$, we have $$P(f,\varphi)=\lim_{n \to + \infty} \frac{1}{n} \log \sum_{y\in f^{-n}(x_0)} \exp(S_n(\varphi)(y));$$
\item
If for every integer $n\geq 1$, put~$\Per_n(f)\=\{p\in J(f): f^n(p)=p\},$ then we have $$P(f,\varphi)=\lim \limits_{n\to +\infty}\frac{1}{n}\log \sum_{p\in \Per_n(f)}\exp(S_n(\varphi)(p)).$$
\item
Let $\nu$ be the unique equilibrium state of $f$ for the potential~$\varphi$ given by Lemma~\ref{l:uniqueeq}. Then for every H\"{o}lder continuous function $\psi:J(f)\to \R$ such that $\varphi+\psi$ is hyperbolic for $f$, we have $$\lim_{n\to +\infty}\frac{1}{n} \log \int \exp(S_n(\psi))d \nu =P(f,\varphi+\psi)-P(f,\varphi).$$
\end{enumerate}
\end{generic}

In order to prove our Main Theorem, we also need the following lemma, which is a variant of a general result of Kifer in~\cite{Kif90}.
\begin{lemm}[Theorem~C,~\cite{ComRiv11}]\label{l:large}
Let $X$ be a compact metrizable topological space, and let $T: X \to X$ be
a continuous map such that the measure-theoretic entropy of $T$, as a function defined
on $\cM(X,T)$ is finite and upper semi-continuous. Fix $\varphi\in C(X)$, and let $\cW$ be a dense
vector subspace of $C(X)$ such that for each $\phi\in \cW$ there is a unique equilibrium state of $T$
for the potential $\varphi+\phi$. Let $I^\varphi: \cM(X)\to [0, +\infty]$ be the function defined by
$$I^\varphi(\mu)=\begin{cases}P(T,\varphi)-\int_X\varphi d\mu- h_\mu(T)
\,\,\,\,\,\,\mbox{if}\,\, \mu\in \cM(X,T);
\\ +\infty \,\,\hspace{35mm} \mbox{if}\,\, \mu\in
\cM(X)\setminus\cM(X,T).\end{cases}$$
Then every sequence $(\Omega_n)_{n\ge 1}$ of Borel probability measures on $\cM(X)$ such that for every $\phi\in \cW$
\begin{equation}\label{e:et}
\lim_{n\to +\infty}\frac{1}{n}\log \int_{\cM(X)}\exp \left(n\int_X \phi d\mu\right)d \Omega_n(\mu) =P(T, \varphi+\phi)-P(T,\varphi),
\end{equation}
 satisfies a
large deviation principle with rate function $I^\varphi,$ and it converges in the weak* topology to the Dirac mass  supported on the unique equilibrium state of $T$ for the
potential $\varphi$. Furthermore, for each convex and open subset $\sG$ of $\cM(X)$ containing some
invariant measure, we have

$$\lim_{n\rightarrow\infty}\frac{1}{n}\log\Omega_n(\sG)=\lim_{n\rightarrow\infty}\frac{1}{n}\log\Omega_n(\overline{\sG})
=-\inf_{\sG}I^\varphi=-\inf_{\overline{\sG}}I^\varphi. $$
\end{lemm}

\begin{proof}[Proof of the Main Theorem assuming the Key Lemma]
Let $\sP$ be the collection of all Lipschitz  continuous function defined on $J(f).$ Then $\sP$ is a dense vector subspace of $C(J(f)).$
By~\cite[Main Theorem]{LiRiv12two} we have that $\varphi$ is hyperbolic for $f$, and that for every Lipschitz  continuous function $\psi: J(f)\to \R$ we have $\varphi+\psi$ is also hyperbolic for $f$. By Lemma~\ref{l:uniqueeq}  there is a unique equilibrium state $\nu_\varphi$ of $f$ for the
potential $\varphi,$ and for every $\psi\in \sP$ there is a unique equilibrium state of $f$ for the
potential $\varphi+\psi.$

In the complex case, the proof of~\cite[Theorem~B]{ComRiv11} can be adapted to yield the Main Theorem, although there it is assumed  that $f$ satisfies the ``topological Collet-Eckmann" condition.

To prove assertions in the real setting, fix $\psi\in \sP$.  For the
sequence $(\Theta_n)_{n\geq 1}$ associated to periodic points we have
\begin{multline*}
\int_{\cM(J(f))} \exp \left(n\int_{J(f)} \psi d\mu \right)d \Theta_n(\mu)
\\=\frac{\sum_{p\in \Per_n }\exp \left(S_n(\varphi)(p)\right)\exp \left(n\int_{J(f)} \psi d F_n(p)\right)}{\sum_{p'\in \Per_n} \exp \left(S_n(\varphi)(p')\right)}\\ =\frac {\sum_{p\in \Per_n} \exp \left(S_n(\varphi+\psi)(p)\right)}{\sum_{p'\in \Per_n} \exp \left(S_n(\varphi)(p')\right)}.
\end{multline*}
Analogously, for the sequence $(\Theta_n(x_0))_{n\geq 1}$ associated to the iterated preimages of a point
$x_0\in J(f)$, we have
$$\int_{\cM(J(f))} \exp \left(n\int_{J(f)} \psi d\mu \right)d \Theta_n(x_0)(\mu)= \frac {\sum_{x\in f^{-n}(x_0)} \exp \left(S_n(\varphi+\psi)(x)\right)}{\sum_{y\in f^{-n}(x_0)} \exp \left(S_n(\varphi)(y)\right)};$$ for the the sequence $(\Sigma_n)_{n\geq 1}$ associated to the Birkhoff averages we have $$\int_{\cM(J(f))} \exp \left(n\int_{J(f)} \psi d\mu \right)d \Sigma_n(\mu)=\int_{J(f)}\exp(S_n(\psi))d \nu_\varphi.$$ Therefore, the Key Lemma implies that~(\ref{e:et}) holds with $\phi=\psi$, and with $(\Omega_n)_{n\geq 1}$ replaced by each of the sequences $(\Theta_n(x_0))_{n\geq 1}$, $(\Theta_n)_{n\geq 1}$ and $(\Sigma_n)_{n\geq 1}$, respectively. On the other hand, the topological entropy of~$f$ is finite, and the measure-theoretic entropy of~$f$ is upper semi-continuous, see for example~\cite{BruKel98, kel98}.
Consequently, the assertion of Main theorem follows from Lemma~\ref{l:large}. The proof is complete.
\end{proof}

\section{Proof of the Key Lemma}\label{s:proofkeylemma}
 Given an interval map~$f$ in $\sA_\R$, and a continuous function $\phi:J(f)\rightarrow \R,$ denote by $\cL_\phi$ the {\em Ruelle-Perron-Frobenius
operator}, acting on the space of bounded
functions defined on $J(f)$ and taking values in $\C$, as follows,
$$\cL_{\phi}(h)(x)\=\sum_{y\in f^{-1}(x)}\exp\left(\phi(y)\right)h(y).$$ Moreover, put $\hcL_\phi\=\exp(-P(f,\phi))\cL_\phi$.

The following proposition is the main ingredient to prove parts~$1$ and~$3$ of the Key Lemma.

\begin{prop}[Proposition~2.1, \cite{LiRiv12one}]\label{p:opertor}
Let $f$ be an interval map in~$\sA_\R$ that is topologically exact on its Julia set $J(f),$ and denoted by $\1$ the function defined on $J(f)$ that is constant equal to $1.$ Then for every H\"{o}lder continuous function $\phi:J(f)\to \R$ that is hyperbolic for $f$,  and every $\eps>0$, we have for every sufficiently large integer~$n$ $$\exp(-\eps n)\leq \inf_{J(f)}\hcL_{\phi}^n(\1)\leq \sup_{J(f)}\hcL_{\phi}^n(\1)\leq \exp(\eps n).$$
\end{prop}

We also need the following lemma.
\begin{lemm}[Lemma~5.1, \cite{LiRiv12one}]\label{l:conj}
Let $f$ be an interval map in $\sA_\R,$ and let $\phi:J(f)\to \R$ be H\"{o}lder continuous that is hyperbolic for~$f$. If $\mu$ is a $\exp(P(f,\phi)-\phi)$-conformal measure for $f$, then for every function~$h\in L^1(\mu),$ we have $$\int_{J(f)} \hcL_\phi (h)\ d\mu =\int_{J(f)} h \ d \mu.$$
\end{lemm}

\begin{proof}[Proof of parts~$1$ and~$3$ of the Key Lemma]
Part~$1$ is a direct consequence of Proposition~\ref{p:opertor} with $\phi=\varphi.$ To prove part~$3$, note that by Lemma~\ref{l:uniqueeq} there is a $\exp(P(f,\varphi)-\varphi)$-conformal measure $\mu$ for $f$. Moreover,~$\nu$ is absolutely continuous with respect to~$\mu$.  Let $h_\nu:J(f)\to \R$ be a density function of $\nu$ with respect to $\mu$, so that $\nu=h_\nu\mu$. By Lemma~\ref{l:conj}, we have
\begin{align*}
\int_{J(f)} \exp(S_n(\psi))d \nu&=\int_{J(f)} \exp(S_n(\psi))h_\nu d \mu= \int_{J(f)} \hcL^n_\varphi(\exp(S_n(\psi))h_\nu)d \mu
\\&=\exp(-nP(f,\varphi) \int_{J(f)} \cL^n_\varphi(\exp(S_n(\psi))h_\nu)d \mu\\& = \exp(-nP(f,\varphi) \int_{J(f)} \cL^n_{\varphi+\psi}(h_\nu)d \mu\\&=\exp(n(P(f,\varphi+\psi)-P(f,\varphi)))\int_{J(f)} \hcL_{\varphi+\psi}(h_\nu)d \mu
\end{align*}
By Proposition~\ref{p:opertor} with $\phi=\varphi+\psi$, for every $\eps>0$ we have for every integer $n\geq 1$ large enough $$\exp(-\eps n)h_\nu \leq \inf_{J(f)}\hcL_{\varphi+\psi}^n(h_\nu)\leq \sup_{J(f)}\hcL_{\varphi+\psi}(h_\nu)\leq \exp(\eps n)h_\nu.$$ It follows that
\begin{multline*}\exp(-\eps n)=\int_{J(f)}\exp(-\eps n)h_\nu d\mu\leq\int_{J(f)} \hcL_{\varphi+\psi}(h_\nu)d \mu\\\leq \int_{J(f)}\exp(-\eps n)h_\nu d\mu=\exp(\eps n).
 \end{multline*}
 Therefore, for every sufficiently large integer~$n$ we have $$P(f,\varphi+\psi)-P(f,\varphi)-\eps\leq \frac{1}{n}\log \int_{J(f)} \exp(S_n(\psi))d\mu_\varphi \leq P(f,\varphi+\psi)-P(f,\varphi)+\eps.$$
Letting $\eps\to 0$ we complete the proof of part~$3$ of the Key Lemma.
\end{proof}

The rest of  this section is devoted to prove part~$2$ of the Key Lemma.
The proof, which is given at the end of this section, depends on several lemmas.

\begin{lemm}\label{l:stable}
 Let $f$ be an interval map in $\sA_\R$
 that is topologically exact on $J(f).$
Then, for every $\kappa > 0$ there is $\delta_1 > 0$ such that the following holds. For every interval $J\subset I$ such that $\partial J \subset J(f)$ and $|J|\leq \delta_1$,
every integer $n \geq 0$, and every pull-back $W$ of $J$ by $f^n$, we have $|W| \leq \kappa.$
\end{lemm}
\begin{proof}See the proof of~\cite[Lemma~A.2]{Riv1206} for example. There it is assumed  that all cycles are hyperbolic repelling, but the proof can be adapted to yield the lemma.
\end{proof}

\begin{lemm}[Lemma~2.3, \cite{LiRiv12one}]\label{l:onetime}
Let $f$ be an interval map in $\sA_\R$ and let $\varphi:J(f)\to \R$ be a H\"{o}lder continuous function. Then for every integer $N\geq 1$ and there is a constant $C>1$ such that the function~$\tvarphi\= \frac{1}{N}S_N(\varphi)$ satisfies the following properties:
\begin{enumerate}[1.]
\item The function~$\tvarphi$ is H\"{o}lder continuous of the same exponent as $\varphi$, $P(f, \tvarphi) = P(f, \varphi)$,  and $\varphi$ and $\tvarphi$ share the same equilibrium states;
\item For every integer $n\geq 1$, we have $$\sup_{J(f)}|S_n(\tvarphi)-S_n(\varphi)|\leq C.$$
\end{enumerate}
\end{lemm}

\begin{lemm}[Lemma~3.2, \cite{LiRiv12two}]
\label{l:almost properness}
For each interval map~$f : I \to I$ in~$\sA_\R$ there is~$\varepsilon > 0$ such that the following property holds.
Let~$J_0$ be an interval contained in~$I$ satisfying~$|J_0| \le \varepsilon$, let~$n \ge 1$ be an integer, and let~$J$ be a pull-back of~$J_0$ by~$f^n$, such that for each~$j$ in~$\{1, \ldots, n \}$ the pull-back of~$J_0$ by~$f^j$ containing~$f^{n - j}(J)$ has length bounded from above by~$\varepsilon$.
If in addition the closure of~$J$ is contained in the interior of~$I$, then~$f^n(\partial J) \subset \partial J_0$.
\end{lemm}

Given an integer~$n \ge 1$ and a point~$x$ in~$\dom (f)$, a preimage~$y$ of~$x$ by~$f^n$ is \emph{critical} if~$Df^n(y) = 0$, and it is \emph{non-critical} otherwise.
The following lemma was proved  in~\cite{LiRiv12one}, whose proof is based on Przytycki and Urba{\'n}ski's adaptation to one-dimensional maps of Katok-Pesin theory, see for example~\cite[\S$11.6$]{PrzUrb10}.
\begin{lemm}[Lemma~2.4, ~\cite{LiRiv12one}]
\label{l:diffeomorphic pressure}
Let~$f$ be an interval map in~$\sA_\R$ that is topologically exact on~$J(f)$, and let~$\phi : J(f) \to \R$ be a H{\"o}lder continuous potential satisfying~$\sup_{J(f)} \phi < P(f, \phi)$.
Then for every point~$x_0$ of~$J(f)$ having infinitely many non-critical preimages, there is~$\delta > 0$ such that the following property holds: If for each integer~$n \ge 1$ we denote by~$\fD_n$ the collection of diffeomorphic pull-backs of~$B(x_0, \delta)$ by~$f^n$, then
$$ \liminf_{n \to + \infty} \frac{1}{n} \log \sum_{W \in \fD_n} \inf_{W \cap J(f)} \exp(S_n(\phi))
\ge
P(f, \phi). $$ In particular, $$\liminf_{n \to + \infty} \frac{1}{n} \log \sum_{x\in f^{-n}(x_0)} \exp(S_n(\phi)(x))
\ge
P(f, \phi).$$
\end{lemm}

  Throughout the rest of this section,  fix~$f$ and $\varphi$ as in the Key Lemma. Recall that for every integer $n\geq 1$,
$\Per_n(f)=\{p\in J(f): f^n(p)=p\}.$ Since $\varphi$ is hyperbolic for $f$, there is an integer~$N \ge 1$ such that the function~$\tvarphi \= \frac{1}{N} S_N(\varphi)$ satisfies~$\sup_{J(f)} \tvarphi < P(f, \varphi)$.
By part~$1$ of Lemma~\ref{l:onetime}, the function~$\tvarphi$ is H\"{o}lder continuous and $$P(f, \tvarphi) = P(f, \varphi)>\sup_{J(f)}\tvarphi.$$

\begin{lemm}\label{l:upperiodic}
$\liminf \limits_{n\to +\infty}\frac{1}{n}\log \sum_{p\in \Per_n(f)}\exp(S_n(\tvarphi)(p))\geq P(f,\tvarphi).$
\end{lemm}
\begin{proof}
Let $x_0$ be a point of $J(f)$ that is not a boundary point of a connected component of $\R\setminus J(f)$. Assume in addition that for every $c$ in $\Crit(f)$ and every integer $n\geq 1$ we have $f^n(c)\neq x_0.$
Let $\delta>0$ be the constant given by Lemma~\ref{l:diffeomorphic pressure} with $\phi=\tvarphi$, and
let $V\subset B(x_0, \delta)$ be a closed interval such that $\partial V\subset J(f)$ and such that~$x_0$ is an interior point of~$V$.
Since $f$ is topologically exact on its Julia set, there is $s\geq 1$ such that $f^{s}(V)\supset J(f).$

For each integer~$n \ge 1$, denote by~$\fD_n$ the collection of diffeomorphic pull-backs of~$B(x_0,\delta)$ by~$f^n$. For every integer $n> s$ and each $D\in \fD_{n-s}$, let $D'$ be the connected component of $f^{-(n-s)}(V)$ contained in $D$. Note that $\partial D' \subset J(f)$ and  $f^{n-s}(D')=V$. Therefore $D'\subset f^n(D'),$ and so there is a point~$p\in \Per_n(f)$ in $D'.$  It follows that for every $n> s$
\begin{multline*}
\frac{1}{n}\log \sum_{p\in \Per_n(f)}\exp(S_n(\tvarphi)(p))\geq \frac{1}{n}\log \left( \sum_{D\in \fD_{n-s}}\sum_{p\in D'\cap\Per_n(f)}\exp(S_n(\tvarphi)(p))\right)\\= \frac{1}{n}\log \left( \sum_{D\in \fD_{n-s}}\sum_{p\in D'\cap\Per_n(f)}\left(\exp\left(S_{n-s}(\tvarphi)(p)+S_s(\tvarphi)(f^{n-s}(p))\right)\right)\right)\\ \geq \frac{1}{n} \log \sum_{D \in \fD_{n-s}} \inf_{D' \cap J(f)} \exp(S_{n-s}(\tvarphi))+\frac{s}{n} \inf_{J(f)} \tvarphi .
\end{multline*}
By Lemma~\ref{l:diffeomorphic pressure} we have $$\liminf_{n \to + \infty} \frac{1}{n} \log \sum_{D \in \fD_n} \inf_{D' \cap J(f)} \exp(S_n(\tvarphi))
\ge
P(f, \tvarphi).$$  Hence,
\begin{equation*}\label{e:below}
\liminf_{n\to +\infty}\frac{1}{n}\log \sum_{p\in \Per_n(f)}\exp(S_n(\tvarphi)(p))\geq P(f,\tvarphi),
\end{equation*}
and we complete the proof.
\end{proof}

\begin{lemm}\label{lem:bad}
Let $\kappa>0$, and
let $L \ge 1$ be an integer such that for every critical point~$c$ and every~$i$ in~$\{1,2,\cdots, L\}$, we have either
$$ f^i(c)\not\in B(\Crit, \kappa)
\text{ or }
f^i(c)\in \Crit(f). $$
Then for every interval $W$ and each integer $n\geq 1$ such that for every $i$ in $\{0,1, \cdots, n\}$ we have $|f^i(W)|<\kappa,$  there are at most $2^{1+n/L}$ critical points of $f^n$ in $W$. In particular, there are at most $ 2^{2+n/L}$ maximal monotonic intervals of $f^n$ in~$W$.
\end{lemm}
\begin{proof}
Fix an interval $W$ and an integer $n\geq 1$, such that for every $i$ in $\{0,1, \cdots, n\}$ we have $|f^i(W)|<\kappa$.  Define an integer $s \ge 1$ and a strictly increasing sequence of integers $(n_0,\cdots, n_s)$ with $n_0=0$ and $n_s=n$, by induction as follows. Suppose $j \ge 0$ is an integer such that $n_j \le n-1$ is already defined. If~$n_j+L \ge n$ or if $n_j+L \le n-1$ and for each $i$ in $\{n_j+L, \cdots, n-1\}$ the set $f^i(W)$ does not intersect~$\Crit(f)$, then put $n_{j+1}\=n, s\=j+1$ and stop. Otherwise, let~$n_{j+1}$ be the least integer $i$ in  $\{n_j+L, \cdots, n-1\}$ such that $f^i(W)\cap \Crit(f)\neq \emptyset$.

Using the definitions of $L$ and our construction of $(n_0, \cdots, n_s)$, we know that  and for every $i\in \{1,\cdots, s\}$ we have $n_i-n_{i-1}\geq L$ and the map $f^{n_{i}-n_{i-1}}$ has at most one critical point in $f^{n_{i-1}}(W)$.  It follows that the map $f^n$ has at most~$2^s$ critical points in~$W$.
Note that $s \le 1+n/L$, we conclude that the number of maximal monotonic intervals of $f^n$ contained in $W$ less than $1+2^s\leq 2^{2+n/L}.$
\end{proof}

\begin{lemm}\label{l:boundperiodic}
For every $\eta>1$ there is $\delta>0$ such that the following holds. Let $V$ be an interval intersecting $J(f)$ and  such that $|V|<\delta$. For every integer $n\geq 1$ and each pull-back $W$ of $V$ by $f^{n}$
there are at most $\eta^n$ periodic points of $f$ of periodic~$n$ in $W$.
\end{lemm}

\begin{proof} Fix $\eta>1$ and let $L>1$ be large enough such that for every $n\geq L$ we have $2^{2+n/L}<\eta^n.$ Let $\kappa>0$ be sufficiently small so that for every critical point~$c$ and every~$i$ in~$\{1,2,\cdots, L\}$, we have either
$$ f^i(c)\not\in B(\Crit(f), \kappa)
\text{ or }
f^i(c)\in \Crit(f). $$
Let $\delta>0$ be the constant given by Lemma~\ref{l:stable}. Let $V$ be an interval intersecting $J(f)$ and such that $|V|<\delta.$
By Lemma~\ref{lem:bad}, for every $n\geq L$ and each pull-back $W$ of $V$ by $f^{n}$
there are at most $2^{2+n/L}$ maximal monotonic intervals of~$f^n$ in~$W$.

To complete the proof, it suffices to prove for every $n\geq 1$ and every interval~$U$ such that $f^n$ is strictly monotonic on $U$, there is at most one point in $U\cap\Per_n(f)$.  Otherwise, if there are two distinct points $p$ and $p'$ in $U\cap\Per_n(f)$, then by assumption $f^n([p,p'])=[p,p'].$ It follows that for any integer $k\geq 1$ we have $f^{nk}([p,p'])=[p,p']\not\supset J(f).$ This is a contradiction with our assumption that~$f$ is topologically exact on its Julia set.  The lemma is proved.
\end{proof}

\begin{proof}[Proof of part~$2$ of the Key Lemma]
In view of part~$2$ of Lemma~\ref{l:onetime} and of Lemma~\ref{l:upperiodic},
 it suffices to show
$$\limsup \limits_{n\to +\infty}\frac{1}{n}\log \sum_{p\in \Per_n(f)}\exp(S_n(\tvarphi)(p))\leq P(f,\tvarphi).$$
Denote by~$\varepsilon_0 > 0$ the constant given by Lemma~\ref{l:almost properness}. Fix $\eps\in (0,\eps_0)$ and let~$\delta_1$ be the constant given by Lemma~\ref{l:stable} with $\kappa=\eps$. Let $\eta>1$ be given, and let $\delta >0$ be the constant given by Lemma~\ref{l:boundperiodic}. Put $r_0\= \min \{\delta_1, \delta \},$ and let $\cV$ be a finite covering of $J(f)$ by intervals such that for every $V$ in $\cV$ we have $|V|<r_0$ and $\partial V\subset J(f).$ Moreover, put $\partial \cV\=\bigcup_{V\in \cV}\partial V.$ In view of~\cite[Lemma~2.8]{LiRiv12two}, see also~\cite[Lemma~4]{Prz90}, there is~$C_@>0$ such that for every integer $n\geq 1$   we have
\begin{equation}\label{e:top}
\sum_{y\in f^{-n}(\partial \cV)} \exp(S_n(\tvarphi)(y))\leq C_@\exp\left(n(P(f,\tvarphi)+\eps)\right).
\end{equation}

Fix an integer $n\geq 1$, and let $\cW_{n}$ be collection of all pull-backs of elements of~$\cV$ by $f^{n}$. For every $W\in \cW_{n}$, by Lemma~\ref{l:boundperiodic} there are at most $\eta^n$ periodic points of~$f$ of periodic $n$ in $W$, and by Lemma~\ref{l:stable} for every $i$ in $\{0,1, \cdots, n-1\}$ we have $|f^i(W)|<\eps.$ Therefore, by Lemma~\ref{l:almost properness} we have $f^{n}(\partial W)\subset \partial V$ if $W\cap \partial I=\emptyset.$  On the other hand, there exist $C>1$ and $\alpha>0$ only dependent of $\tvarphi$ such that for every $x,y\in W$ we have $|S_n(\tvarphi)(x)-S_n(\tvarphi)(y)|\leq C n \eps^\alpha.$ It follows that
\begin{multline*}
\sum_{p\in \Per_{n}(f)}\exp(S_n(\tvarphi)(p))\\= \sum_{W\in \cW_{n},\atop W\cap \partial I=\emptyset} \sum_{p\in W\cap\Per_{n}(f)}\exp\left(S_n(\tvarphi)(p)\right)+ \sum_{W\in \cW_{n},\atop W\cap \partial I\neq \emptyset}\sum_{p\in W\cap \Per_{n}(f)}\exp\left(S_n(\tvarphi)(p)\right)\\\leq  2 \#\cV \eta^n \exp\left(n\sup_{J(f)} \tvarphi\right)+2\eta^n\sum_{y\in f^{-n}(\partial \cV)} \exp\left(S_{n}(\tvarphi)(y)+ C n\eps^\alpha \right)
\end{multline*}
Together with $\sup_{J(f)}\tvarphi <P(f,\tvarphi)$ and~(\ref{e:top}),  this implies
\begin{multline*}
\sum_{p\in \Per_{n}(f)}\exp(S_n(\tvarphi)(p))\\<  2 \#\cV \eta^n \exp(n P(f,\tvarphi))+ 2C_@\eta^n \exp(C n \eps^\alpha) \exp(n (P(f,\tvarphi)+\eps))\\\leq 2(\#\cV+C_@)\eta^n \exp(n(P(f,\tvarphi)+C\eps^\alpha+\eps)).
\end{multline*}
Hence,
$$\limsup_{n\to +\infty}\frac{1}{n}\log \sum_{p\in \Per_{n}(f)}\exp(S_n(\tvarphi)(p)) \leq \log \eta+ P(f,\tvarphi)+C\eps^\alpha+\eps.$$
Letting $\eta \to 1$ and $\eps\to 0$, we have
$$\limsup_{n\to +\infty}\frac{1}{n}\log \sum_{p\in \Per_{n}(f)}\exp(S_n(\tvarphi)(p)) \leq P(f,\tvarphi),$$ and the proof is complete.
\end{proof}


\end{document}